\documentclass[a4wide,11pt]{amsart}

\usepackage{amsmath}
\usepackage{amssymb}
\usepackage{amsfonts}
\usepackage{subfig}
\usepackage{graphicx}
\usepackage{yfonts}

\usepackage{epic}
\usepackage{latexsym}

\usepackage{epsfig}

\usepackage[mathscr]{euscript}

\usepackage[all]{xy}

\DeclareMathAlphabet{\mathpzc}{OT1}{pzc}{m}{it}

\newtheorem{lemma}{Lemma}[section]
\newtheorem{theorem}[lemma]{Theorem}

\newcommand{\cl}{{\mathcal L}}

\newcommand{\ch}{{X \choose 2}}

\newcommand{\cC}{{\mathcal C}}

\newcommand{\cL}{{\mathcal L}}

\newcommand{\cR}{{\mathcal R}}

\begin{document}

\title[Tree Reconstruction from Triplet Cover Distances]
{Reconstructing fully-resolved trees from triplet cover distances}
\author{Katharina T. Huber and Mike Steel}
\address{
(K.T.H.): School of Computing Sciences,
University of East Anglia, UK; phone: +44 (0) 1603 593211;
FAX: +44 (0) 1603 593345.
 \\
(M.S.): Department of Mathematics and Statistics,
University of Canterbury, New Zealand.
}

\date{\today}
\maketitle

\begin{abstract}
It is a classical result that any finite tree with positively 
weighted edges, and without vertices of degree 2, is uniquely determined
by the weighted path distance between each pair of leaves. Moreover, it 
is possible for a (small) strict subset $\cl$ of leaf pairs to 
suffice for reconstructing the
tree and  its edge weights, given just the distances between 
the leaf pairs in $\cl$. 
It is known that any set $\cl$ with this property for a tree in 
which all interior vertices have degree 3 must form a {\em cover} 
for $T$ -- that is, for each
interior vertex $v$ of $T$, $\cl$ must contain a pair of leaves 
from each pair of the three components of  $T-v$.  Here we 
provide a partial converse of this
result by showing that if a set $\cl$ of leaf pairs forms 
a cover  of a certain type for such a tree $T$ then $T$ 
and its edge weights can be uniquely determined from
the distances between the pairs of leaves in $\cl$. Moreover, 
there is a polynomial-time algorithm for achieving this 
reconstruction. The result establishes a special
case of a recent question concerning `triplet covers', 
and is relevant to a problem arising in evolutionary genomics.

\end{abstract}


{\bf Email:} katharina.huber@cmp.uea.ac.uk (Corresponding author)

\section{Introduction}\label{intro}

Any tree $T$ with positively weighted edges, induces 
a metric $d$ on the set of leaves by considering the 
weighted path distance in $T$ between each pair of leaves.
Moreover, provided $T$ has no vertices of degree 2, 
and that we ignore the labeling of interior vertices, 
both $T$ and its edge weights are uniquely 
determined by the metric $d$.  This uniqueness result 
has been known since the $1960$s and fast algorithms 
exist for reconstructing both the tree and its edge weights 
from $d$ (for further background the interested  
reader may consult ~\cite{bar} and \cite{sem} and the references therein).  

The uniqueness result and the algorithms are important 
in evolutionary biology for reconstructing an evolutionary 
tree of species from genetic data \cite{fel}. However  
in this setting one frequently may not have $d$-values 
available for all pairs of species, due to the patchy 
nature of genomic coverage \cite{san}.  

This raises a fundamental mathematical question -- for 
which subsets of pairs of leaves of a tree do we need 
to know the $d$-values in order to uniquely recover the 
tree and its edge weights? In general this appears a 
difficult question (indeed determining 
whether such a partial $d$-metric is realized by {\em any} 
tree is NP-hard \cite{far}).  However, some sufficient 
conditions (as well as some necessary conditions) for 
uniqueness to hold have been found, in  \cite{cha,yus},  
and more recently in  \cite{dre3}, and \cite{gue}.    In this 
paper we consider the uniqueness question for trees that are 
`fully-resolved' (i.e. all the interior vertices have degree 3) 
as these trees are of particular importance in evolutionary 
biology, and because the uniqueness question is easier to 
study for this class of trees.   

The structure of this paper is as follows. First we introduce 
some background terminology and concepts, and then we define 
the particular type of subsets of leaf pairs (called `stable 
triplet covers') which we show suffice to uniquely determine 
a fully-resolved tree. Moreover, we show how 
this comes about by establishing two combinatorial properties 
of stable triplet covers - a `shellability' property and a 
graph-theoretic property related to tree-width, which we show 
is quite different to shellability.  We  conclude by providing 
a proof that a polynomial-time algorithm will reconstruct a 
tree and its edge weights for any set of leaf pairs that 
contains  a stable triplet cover (or more generally a shellable 
subset).  Our result answers a special case of the question posed 
at the end of \cite{dre3} of whether every `triplet cover' of a 
fully-resolved tree determines
the tree and its edge weights.

\section{Preliminaries}
We now introduce some precise definitions required to state and prove our main results.
We mostly follow the  notation and terminology of \cite{sem} and \cite{dre3}.

\subsection{$X-$trees, edge-weightings and  distances}

For the rest of the paper, assume that $|X|\geq 3$. 
An $X-$tree $T=(V,E)$ is 
a graph theoretical tree whose leaf set is $X$ and which does not have any
vertices of degree 2. We call an $X-$tree {\em fully-resolved} if 
every {\em interior vertex} of $T$, that is, every non-leaf vertex of $T$,
has degree three. Moreover, we call two distinct leaves $x$ and $y$ 
of $T$ a {\em cherry} of $T$, denoted by $x,y$, if the parent of $x$ is
simultaneously the parent of $y$. For any subset $Y\subseteq X$, we denote
by $T|Y$ the $Y$-tree obtained by restricting $T$ to $Y$ (suppressing resulting
degree two vertices).    

An example of a fully-resolved $X-$tree for $X=\{a,b,c,d,e,f,g\}$, 
and having two cherries, is shown in Fig. \ref{figure1}(i).

In case $|Y|=4$, say $Y=\{a,b,c,d\}$,
and the path from $a$ to $b$ does not share a vertex with the path from
$c$ to $d$ in $T|Y$, we refer to $T|Y$ as a {\em quartet tree} and
denote it by $ab||cd$. Note that by deleting any edge $e\in E$ from  $T$ 
the leaf sets $A_e$ and $B_e:=X-A_e$ of the resulting two trees induce 
a bipartition of $X$.
We refer to such a bipartition as {\em $X-$split} and denote it by $A|B$
where $A:=A_e$ and $B:=B_e$ and $A_e$ and $B_e$ are as above.
We say that two $X-$trees $T=(V,E)$ and $T'=(V',E')$ are {\em equivalent}
if there exists a bijection $\phi:V\to V'$ that is the identity on $X$
and extends to a graph isomorphism from $T$ to $T'$.

Suppose for the following that $T=(V,E)$ is an $X-$tree. Then we call
a map $w:E\to \mathbb R_{\geq 0}$ that assigns a {\em weight}, that
is, a non-negative real number, to
every edge of $T$ an {\em edge-weighting} for $T$. 
Note that this definition implies that some of the
edges of $T$ might have weight zero. We denote an $X-$tree $T$ together
with an edge-weighting $w$ by the pair $(T,w)$ and call an edge-weighting that
assign non-zero weight to every edge of $T$ that is not incident with
a leaf of $T$ {\em proper}. Note that for any edge-weighting $w$ of $T$,
taking the sum of the weights of the edges on the shortest path from some 
$x\in X$ to some $y\in X$ induces a distance between $x$ and $y$ and thus 
a distance $d=d_{(T,w)}$ on $X$.  

For example, in the tree in Fig. \ref{figure1}(i), 
if each edge has weight 1, then $d(a,b)=2, d(c,e)=4$, and $d(c,f)=5$.

 \begin{figure}[h] \begin{center}
\resizebox{12cm}{!}{
\includegraphics{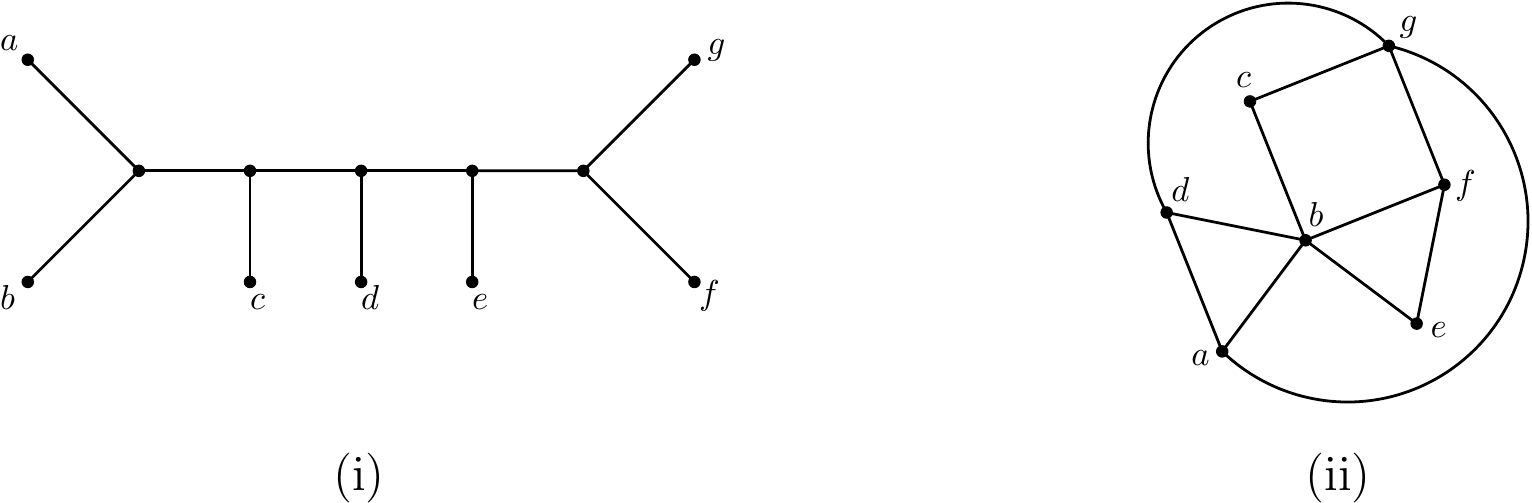}
}
\caption{\label{figure1}
(i) A fully-resolved tree $X-$tree $T$ for $X=\{a,b,c,d,e,f,g\}$; (ii) the graph $(X, \cl)$ corresponding to a strong lasso $\cl$ for $T$ (discussed further in Example 1).}
\end{center}

\end{figure}

\subsection{Lassos}

We call a subset
of $X$ of size two a {\em cord} of $X$ and, for $a,b\in X$ distinct
write $ab$ rather than $\{a,b\}$ for the cord containing $a$ and $b$.
Also, for any non-empty set $\cl\subseteq {X\choose 2}$ of cords of $X$,
we denote the edges of the graph $(X,\cl)$ whose vertex set
is $X$ and whose edge set is the set $\{\{a,b\}: ab\in\cl\}$ by
$ab$ rather than $\{a,b\}$, $ab\in \cl$. 

Suppose for the following that 
$\cl\subseteq {X\choose 2}$ is a non-empty set of cords of $X$.
If $T'=(V',E')$ is a further $X-$tree and $w$ and $w'$
are edge-weightings for $T$ and $T'$, respectively,
such that $d_{(T,w)}(x,y)=d_{(T',w')}(x,y)$ holds for all $xy\in \cl$
then we say that $(T,w)$ and $(T',w')$ 
are {\em $\cl$-isometric}. Moreover
we say that $\cl$ is
\begin{enumerate}
\item[(i)]
an {\em edge-weight lasso} for $T$ if for any two
proper edge-weightings $w$ and $w'$ for $T$ such that
$(T,w)$ and $(T,w')$ are $\cl$-isometric we have that $w=w'$.
\item[(ii)]
a {\em topological lasso} for $T$ if for any other
$X-$tree $T'$ and any two
proper edge-weightings $w$ and $w'$ for $T$ and $T'$, respectively, 
such that
$(T,w)$ and $(T',w')$ are $\cl$-isometric we have that $T$ and $T'$ are
equivalent.
\item[(iii)]
a {\em strong lasso} for $T$ if $\cl$ is simultaneously an
edge-weight and a topological lasso for $T$.
\end{enumerate}
If $\cl$ is a strong lasso for an $X-$tree then the graph $(X, \cl)$ must be connected, and each component of this graph must be
non-bipartite \cite{dre3}.  An example of a strong lasso $\cl$ of the tree in
 Fig. \ref{figure1}(i) is  the set of cords corresponding to the edges of  the graph in  Fig. \ref{figure1}(ii).

\subsection{Shellability}

Given a subset $\cl$ of $\ch$ with $X=\bigcup \cl$, and an $X-$tree 
$T$, we say that $\ch-\cl$ is {\em $T$--shellable} if
there exists an ordering of the cords in $\ch-\cl$ as, say, 
$a_1b_1, a_2b_2, \dots, a_mb_m$ such that, for every $\mu\in \{1,2,\dots,m\}$, 
there exists a pair $x_\mu,y_\mu$ of
 `pivots' for $a_\mu b_\mu$, i.e.,  two distinct elements 
$x_\mu,y_\mu\in X-\{a_\mu,b_\mu\}$,
 for which the tree $T|{Y_\mu}$ obtained from $T$ by restriction
to $Y_\mu:=\{a_\mu,b_\mu,x_\mu,y_\mu\}$, is the quartet tree 
$a_\mu x_\mu||y_\mu b_\mu$, and
all cords in $\binom{Y_\mu}{2}$ except $a_\mu b_\mu$ are contained
in $\cl_\mu:=\cl \cup \big\{a_{\mu'}b_{\mu'}: \mu'\in
\{1,2,\dots,\mu-1\}\big\}$.
Any such ordering of $\ch-\cl$ will also be called a 
{\em shellable ordering} 
 of  $\ch-\cl$, and any subset $\cl$ of $\ch$ for which a 
 shellable ordering of  $\ch-\cl$ exists will  also be called an 
{\em shellable lasso for  $T$}. In \cite[Theorem 6]{dre3}, 
it was established that every shellable lasso for an $X-$tree is
in particular a strong lasso for that tree.

\subsection{Example 1}
Consider the seven-taxon tree, shown in Fig. \ref{figure1}(i), and the lasso $\cl = \{ab, bd, ad, bc, bf, ag, dg, eb, ef, fg, gc\}$
(the edges of the graph in  Fig. \ref{figure1}(ii)).
The  remaining ten chords in $\binom{X}{2}-\cl$ have a shellable ordering, described as follows:
$$bg, cd, ac, cf, ce, af, df, ae, eg, ed,$$
where the corresponding cord pivots are:
$$(a,d), (b,g), (b,d), (b,g), (b,f), (b,g), (b,g), (b,f), (a,f), (b,f),$$
and so $\cl$ is a shellable (and hence strong) lasso for $T$.
\hfill$\Box$

\subsection{Covers, triplet covers}
A necessary condition for $\cl \subseteq \binom{X}{2}$ to be a edge-weight lasso or a topological lasso 
for a fully-resolved
$X-$tree is that $\cl$ forms a  {\em cover} for $T$ -- that is, for each
interior vertex $v$ of $T$, $\cl$ contains a pair of leaves from each pair of the three components of  $T-v$. 
However this condition is not sufficient for $\cl$ to be either an edge-weight lasso or a topological lasso (examples are given
in \cite{dre3}).

A particular type of cover for a fully-resolved $X-$tree 
is a {\em triplet cover} which is defined as any subset 
$\cl$ of $\binom{X}{2}$ with the 
property that for each interior vertex $v$ of $T$ we can select leaves $a,b,c$ from each of the three components of $T-v$ so that
$ab, ac, bc \in \cl$.  It can be shown that if $\cl$ is a triplet cover for a fully-resolved $X-$tree $T$ then $\cl$ is an edge-weight lasso.
However it is not known whether or not every triplet cover of every such $T$ is also a topological (and thereby a strong) lasso for $T$.

\section{A  special class of triplet covers}
  
Suppose that $T=(V,E)$ is a fully-resolved $X-$tree, and let 
$$
{\rm clus}(T) := \bigcup_{e \in E} \{A_e, X-A_e\},
$$
where $A_e|(X-A_e)$ denotes the $X-$split associated with edge $e\in E$.  
We call the elements in ${\rm clus}(T)$ `clusters'  (in biology, they are 
also sometimes referred to as `clans' \cite{wil}).  Thus a cluster is a 
subset of $X$ that corresponds to the leaf labels on one side of some 
edge of $T$.

Given a collection $\cC$ of non-empty subsets of $X$ we say that any function $f: \cC \rightarrow X$ is a {\em stable transversal} for $\cC$ if it satisfies the two properties:

\begin{itemize}
\item (transversality)  $f(A) \in A$, for all $A \in  \cC$;
\item (stability)    
$f(A) \in B \subseteq A \Longrightarrow f(A) = f(B)$ for all $A, B \in  \cC$.
\end{itemize}

Mostly we will be concerned with stable transversals for ${\rm clus}(T)$, which were introduced in \cite{boc}, though for a different purpose. 
 
 \bigskip
 
\subsection{Example 2} An example of a stable transversals for ${\rm clus}(T)$ is as follows: Consider {\em any} stable transversal $g$ for $2^X$ (equivalently, the function $g(A) = \min A$ under some total ordering of $X$), and consider {\em any} proper edge weighting $w$ of $T$.   For a cluster $A \in {\rm clus}(T)$, consider the subset $A_w$ of leaves of $T$ in 
$A$ that are a closest to the edge $e$ whose deletion induces the split $A|X-A$.  Here `closest' refers to the path distance in $T$ from each leaf in $A$ to $e$ under the edge weighting $w$.  If we let  $f(A) = g(A_w)$,  for each $A \in {\rm clus}(T)$ then $f$ is a stable transversal for ${\rm clus}(T)$. Notice that this holds also for the corresponding function in which `closest' is replaced by `furthest' throughout. 
\hfill$\Box$

\subsection{Example 3} 
Consider the fully-resolved $X-$tree shown in Fig. \ref{figure2}(i), 
and the function $f$ defined as follows:
$f(\{x\}) = x$ for all $x$ in $X$, and
$$f(\{a,a'\})=a, f(\{b,b'\})=b, f(\{c,c'\})=c$$ and
$$f(X-\{a,a'\})=b, f(X-\{b,b'\})=c, f(X-\{c,c'\})=a.$$ 
Then $f$ is a stable transversal for $T$.   Note that the choices of $b,c,a$ in the last line could be replaced by, for example, $c,a,b$ or $c,c,a$ and we would still have a stable transveral.  
\hfill$\Box$

 \begin{figure}[h]
  \begin{center}
\resizebox{12cm}{!}{
\includegraphics{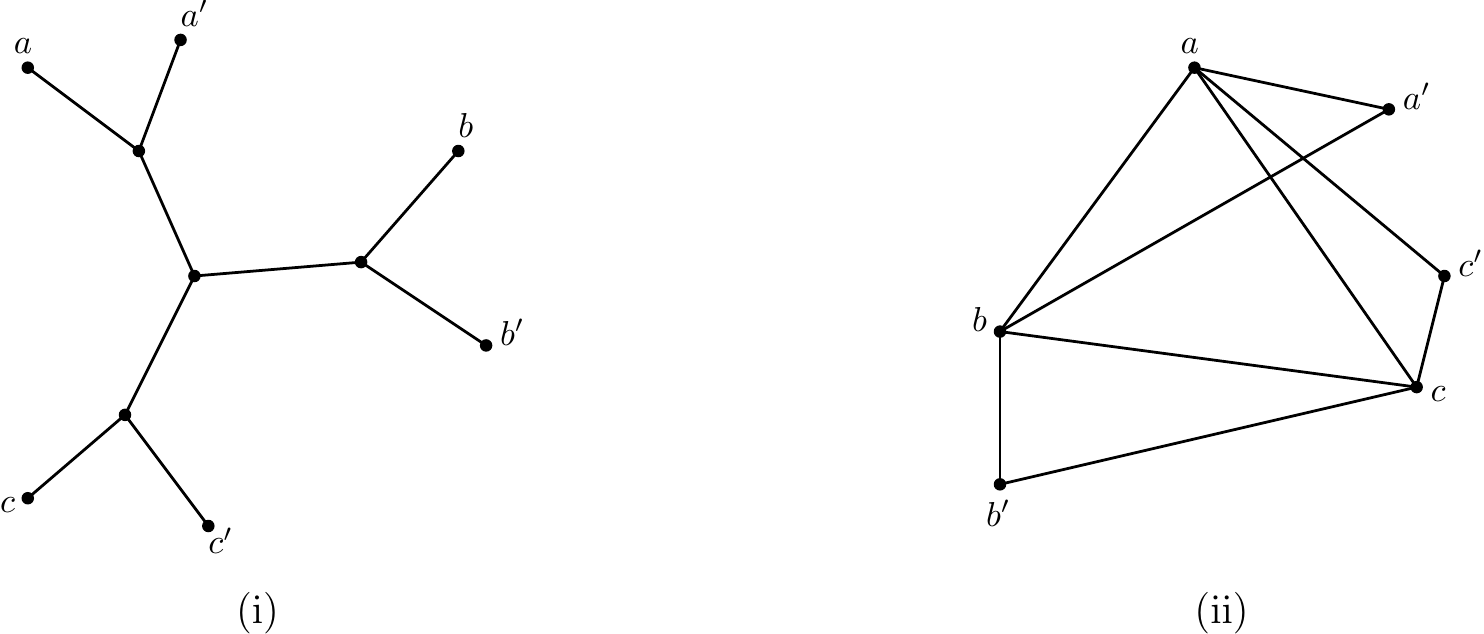}}
\caption{\label{figure2}
(i) A fully-resolved $X-$tree for the set $X=~\{a,a', b,b', c,c'\}$;  the graph $(X,  \cl)$ where $\cl = \cl_{(T, f)}$ forms a stable triplet cover for $T$, and where $f$ is as defined in Example 3.}
\end{center}
\end{figure}

 \section{Stable triplet covers are minimal strong lassos for $T$}
Given a fully-resolved $X-$tree $T$, a stable transversal $f$ of ${\rm clus}(T)$  defines a triplet cover for $T$ as follows:  For each interior vertex $v$ of $T$,  consider the three components of the graph $T-v$, and let
 $A^1_v, A^2_v, A^3_v$ denote their leaf sets.  Then let 
 $$
\cl_{(T, f)} := \bigcup_{v \in V_{\rm int}} \{  f(A^1_v)f(A^2_v), f(A^2_v)f(A^3_v), f(A^3_v)f(A^1_v)\}
$$
where $V_{\rm int}$ denotes the set of interior vertices of $T$.
We say that $\cl$ is a  {\em stable triplet cover} (generated by $f$)  if $\cl = \cl_{(T,f)}$ for some stable transversal $f$ of ${\rm clus}(T)$. 
For example, for the pair $(T, f)$ described in Example 3, we have:
$$\cl_{(T, f)} = \{ab, ac, bc, aa', a'b, bb' ,b'c, cc', c'a\},$$
and the graph $(X, \cl)$ for $\cl=\cl_{(T, f)}$ is 
shown in Fig. \ref{figure2}(ii). 
Notice that not all triplet covers are stable; indeed the set of  triplet covers of a fully-resolved $X-$tree $T$ is precisely the set of subsets of $\binom{X}{2}$ of the form
$\cl_{(T, f)}$ where $f$ is required to  satisfy only the transversality property above for some $f: {\rm clus}(T) \rightarrow X$.

\subsection{2d-trees}
Interestingly, Fig. \ref{figure2}(ii) shows 
that for the set $\cl=\cl_{(T, f)}$ with $T$ and $f$ from Example 3,
the graph  $(X,\cl)$ is a  $2d$-tree, 
where a graph $G=(V,E)$
is called a {\em $2d$-tree} if  there exists an
ordering $x_1,x_2,\ldots, x_n$ of $V$ such that $\{x_1,x_2\}\in E$
and, for $i=3,\ldots, n$ the vertex $x_i$ has degree 2 in the subgraph
 of $G$ induced by $\{x_1,x_2,\ldots, x_i\}$. $2d$-trees are examples
of $kd$-trees which were characterized in \cite{todd} and also 
studied in e.\,g.\,\cite{gue}.   In Fig. \ref{figure2}(ii) an acceptable vertex ordering 
is $a,b,c, a',b',c'$. The graph in Fig. \ref{figure1}(ii) is also a 2d-tree, as can be seen
by considering the vertex ordering $a,b,d, g,c, f,e$. 

\subsection{Main result}

We can now state our first main result which relates stable triplet covers with 2d-trees and shellable lassos.

 \begin{theorem}
 \label{maintheorem}
If $\cl$ is a stable triplet cover of a fully-resolved $X-$tree $T$
with $n:=|X|\geq 3$, then 
\begin{itemize}
\item[(i)]  $(X,\cl)$ is a 2d-tree.
\item[(ii)]  $\cl$ is an shellable lasso for $T$, and so $\cl$ is a strong lasso for $T$. 
\item[(iii)]  $|\cl|= 2n-3$, and so $\cl$ is also a minimal strong lasso for $T$.
\end{itemize}
\end{theorem}

\begin{proof}
We prove parts (i)--(iii) simultaneously by induction on $n=|X|$.  
Shellability holds trivially for $n=3$  (since then $\binom{X}{2}-\cl = \emptyset$), so suppose that it holds when $n=k \geq 3$, and that $T$ is a fully-resolved tree with $k+1$ leaves, and that $\cl$ is a triplet cover for $T$ generated
by a stable transversal $f$ of ${\rm clus}(T)$.  Select
any  cherry $x,y$ of $T$. Without loss of generality, we may suppose that $f(\{x,y\})=x$. 
 Let $$z := f(X-\{x,y\}),   X': = X-\{y\}, T':=T|X',  \cl':= \cl-\{xy, yz\},$$ and define $f': {\rm clus}(T) \rightarrow X$ by setting 
 $$f'(A) =
 \begin{cases}
 & f(A),  \mbox{ if } x \not\in A; \\
 & f(A\cup \{y\}),   \mbox{ if } x \in A.
\end{cases}
$$
Note that, since $f$ is a stable transversal for ${\rm clus}(T)$, it follows that $y$ is not an element of any cord of $\cl'$, and so $\cl' \subseteq \binom{X'}{2}$. Moreover, $y\not=f'(A)$ for
any $A \in {\rm clus}(T')$, and so $f': {\rm clus}(T') \rightarrow X'$. It can now be checked that  $f'$ is a stable transversal for ${\rm clus}(T')$ and so  $\cl'$ is a stable triplet cover of $T'$, generated by $f'$. By the inductive hypothesis (applied to $T'$ and $ \cl'$) it follows 
with regards to (i) that $(X',\cl')$ is a 2d-tree. Clearly
adding $y$ to the vertex set of that graph and $xy$ and $zy$
to its edge set preserves the 2d-tree property. By the definition of
$\cl'$ it is easy to see that the resulting graph is $(X,\cl)$.

Note that regarding (ii) and (iii)
the induction hypothesis implies 
that $|\cl'| = 2k-3$,  and so $|\cl| = 2(k+1)-3$ and that  $\binom{X'}{2} - \cl'$ is shellable. So let us fix an ordering of $\binom{X'}{2} - \cl'$
that provides such a shelling.  This will form the initial segment of a
shellable ordering  of $\binom{X}{2}-\cl$.

To describe this extended ordering, let $v$ be the interior vertex of $T$ adjacent to leaves $x$ and $y$, and let $u$ be the interior vertex of $T$ adjacent to $v$.  Consider the three components of the graph  $T-u$. One component contains $x,y$, and we will denote the leaf sets of the other two components by $X_2$ and $X_3$, where, without loss of generality, $z \in X_3$.  Notice that $\binom{X}{2}-\cl$ is the disjoint union of the three sets: 
 $$\binom{X'}{2}-\cl', \{ty: t \in X_2\} \mbox{ and } \{ty: t \in X_3-\{z\}\}.$$   We order $\binom{X}{2}-\cl$ as follows: the elements of $\binom{X'}{2}-\cl'$ come first, ordered by their shellable ordering, followed by the elements 
$ty$ with  $t \in X_2$ (in any order), followed by the
elements $ty$ with  $t \in X_3-\{z\}$ (in any order).

We claim that any such ordering provides a shellable ordering of $\binom{X}{2}-\cl$. To see this, observe first that, for any leaf $t \in X_2$, the elements $x,z$ provide `pivots' for the pair $t,y$, since
$T|\{x,y,z,t\} = xy||zt$ and all cords in $\binom{\{x,y,z,t\}}{2}$ except $ty$ are contained in $\cl \cup (\binom{X'}{2}- \cl')$.  
Also,  for any  leaf $t \in X_3$,  if we select any leaf $z' \in X_2$ then the pair $x,z'$ provides a `pivot' for $t,y$, since $T|\{x,y,z',t\} = xy||z't$, and all cords in $\binom{\{x,y,z',t\}}{2}$ except $ty$ are contained in $\cl \cup (\binom{X'}{2}- \cl') \cup \{t'y: t' \in X_2\}$.  
In all cases, the cords required for pivoting come earlier in the ordering. 

Thus, we have established that $\cl'$ is an shellable lasso for $T$, and so, by Theorem 6 of \cite{dre3},  $\cl$ is also a strong lasso for $T$.  Moreover, we showed
that $|\cl| = 2|X|-3$, and since this equals the number of edges in any fully-resolved $X$--tree, linear algebra ensures that no strict subset of $\cl'$ could be an edge weight-lasso for $T$. Hence,  $\cl$ is a minimal strong lasso for $T$, which completes the proof of the induction step, and thereby of the theorem.

\end{proof}

\subsection{Remarks}
\begin{itemize}
\item[(1)] Just because a graph $(X, \cl)$ is a 2d-tree, it does not follow that $\cl$ forms a strong (let alone a shellable) lasso for every given fully-resolved $X-$tree $T$.   A simple example is furnished by $X=\{a,b,c,d\}$ and $\cl = \{ab, ac, bc, ad, bd\}$, for which $(X, \cl)$ is a 2d-tree, and yet $\cl$ fails to be a strong lasso for $T= ab||cd$.   

However, if $(X, \cl)$ forms a 2d-tree, or more generally 
if $\cl$ contains a subset $\cl'$ such that $(X, \cl')$ is a 
2d-tree, then $\cl$ is a strong lasso for at least one 
fully-resolved $X-$tree.  The proof is constructive based 
on the ordering $x_1, x_2, \ldots, x_n$ in the definition of 
a 2d-tree:  Start with the tree consisting of leaves $x_1$ 
and $x_2$, and construct a fully-resolved tree as follows: 
for each $i>2$,  if $x_i$ is adjacent to $x_j$ and  $x_k$ in
$(X,\cL')$ (where $j,k<i$) then let $x_i$ be the leaf that is 
attached by a new edge to a {\em new}  subdivision vertex 
on the path connecting $x_j$ and $x_k$ in the tree so-far constructed. 

It may be of interest to explore further the connection between shellability and 2d-trees, and in particular, the question of when the former property for some set $\cl \subseteq \binom{X}{2}$ entails the latter property for $\cl$, or for some subset of $\cl$. 

\item[(2)]    Suppose that $T$ is a fully-resolved $X-$tree, and $\cl \subseteq \binom{X}{2}$ contains a stable triplet cover.  A natural setting in which this situation arises is the following.  Suppose $(T,w)$ is a properly edge-weighted fully-resolved $X-$tree, and $\cl \subseteq \binom{X}{2}$ has the property that, for any interior vertex, $v$, $\cl$ contains every chord $xy$ for which $x$ is a closest leaf to $v$ in one subtree of $T-v$ and $y$ is a closest leaf to $v$ in another subtree of $T-v$.  Then, as noted in Example 2 above, $\cl$ contains a stable triplet cover.   

Now, when $\cl$ contains a stable triplet cover for $T$, it 
follows by  Theorem~\ref{maintheorem} that $\cl$ is a shellable, 
and thereby also a strong lasso for $T$ (since any superset of a
strong lasso for a tree is also a strong lasso for that tree).   
However, it is perhaps not clear how one might efficiently 
construct $(T,w)$ from the distances induced by $\cl$, particularly 
when the subset of $\cl$ corresponding to the stable triplet cover 
is not also given explicitly. Thus, in the next section we describe 
a polynomial-time algorithm for reconstructing $(T, w)$ whenever 
$\cl$ contains some (unknown) shellable lasso for $T$.

\end{itemize}

\section{An algorithm for reconstructing $(T, w)$ from $d_{(T,w)}|\cl$ when $\cl$ contains an shellable lasso for $T$.}

Suppose that $\cl\subseteq {X\choose 2}$ and that $T$ is a fully-resolved $X-$tree, $w$ is a proper 
edge-weighting of $T$ and $d=d_{(T,w)}$.  Starting with $\cl^*= \cl$
add cords to $\cl^*$ and extend the domain of $d$ to those cords, by repeated application of the following extension rule ($\cR$), described in \cite{gue} (Section 6.2, page 246):

\begin{itemize}
\item[($\cR$)]
Whenever $x,y,z,u \in X$ and 
$$
\binom{\{x,y,u,z\}}{2} - \{xz\} \subseteq \cl^*, xz \not\in \cl^*,  
\mbox{ and }
$$ 
$$ 
d(x,y)+d(u,z) < d(x,u)+d(y,z)
$$
add $xz$ to $\cl^*$, and let $d(x,z) := d(x,u)+d(y,z)-d(y,u).$
\end{itemize}
Let ${\rm cl}_\cR(\cl)$ be the set of resulting set of cords obtained from the initial set $\cl$ when this extension rule no longer yields any new cords.  

Note that ${\rm cl}_\cR(\cl)$ can be computed in polynomial time, and that $d-$values are assigned for all cords in ${\rm cl}_\cR(\cl)$. Moreover, if ${\rm cl}_\cR(\cl) = \binom{X}{2}$, then ${\rm cl}_\cR(\cl)$ is a strong lasso for $T$, however the converse does not hold (Example 6.2 of \cite{dre3} provides a counterexample).

\begin{theorem}
If $\cl\subseteq {X\choose 2}$ contains an shellable lasso for a 
fully-resolved $X-$tree  $T$, and $d= d_{(T,w)}$, for some proper 
edge weighting $w$, then ${\rm cl}_\cR(\cl) = \binom{X}{2}$. Consequently,  
$T$  and $w$ can be reconstructed in polynomial time from the 
restriction of $d$ to $\cl$. 
\end{theorem}

\begin{proof}
Suppose that $\cl' \subseteq \cl$ is a  shellable lasso for $T$; we will show that  ${\rm cl}_\cR(\cl') = \binom{X}{2}$ and so
 ${\rm cl}_\cR(\cl) = \binom{X}{2}$.  Suppose to the contrary that  ${\rm cl}_\cR(\cl')$ is a strict subset of $\binom{X}{2}$, and
consider any shelling $a_1b_1, \ldots, a_mb_m$ of the cords in $\binom{X}{2}-\cl'$ (such a shelling exists  by the assumption that
$\cl'$ is an shellable lasso for $T$). Let $j \in \{1, \ldots, m\}$ be the smallest index for which $a_jb_j \not\in {\rm cl}_\cR(\cl')$.  Then the condition on the shelling
ensures that there exists pivots $x_j, y_j \in X-\{a_j,b_j\}$ so that for $Y=\{a_j,b_j,x_j,y_j\} $ we have  $T|Y$ is the quartet tree $a_jx_j||b_jy_j$ and that each cord in $\binom{Y}{2} -\{a_jb_j\}$
either is an element of $\cl'$ or it occurs earlier in the ordering for the shelling than $a_jb_j$, and so, by the minimality assumption concerning $j$,  all these cords lie in ${\rm cl}_\cR(\cl')$.  Consequently,
$a_jb_j \in {\rm cl}_\cR({\rm cl}_\cR(\cl')) = {\rm cl}_\cR(\cl')$, a contradiction.  Thus, our assumption that ${\rm cl}_\cR(\cl')$ is a strict subset of $\binom{X}{2}$ is 
not possible, as required.   

Finally, to efficiently recover $(T, w)$, once $d$ has been defined on all of $\binom{X}{2}$, one can apply standard distance-based reconstruction methods
for fully-resolved trees, such as the Neighbor-Joining method \cite{fel}. 
\end{proof}

\section*{{\sc Acknowledgments}}
{\em
M.S. thanks the Royal Society of NZ under its James Cook Fellowship scheme.  
Both authors thank the organizers of the ``Structure Discovery in Biology: 
Motifs, Networks $\&$ Phylogenies'', Schloss Dagstuhl, Germany, meeting where
this paper was conceived.
 }

\end{document}